\documentclass[reqno]{amsart}

\usepackage{tikz-cd}
\usepackage[margin=1.3in]{geometry}  % set the margins to 1in on all sides
\usepackage{graphicx}              % to include figures
\usepackage{amsmath}               % great math stuff
\usepackage{amsfonts}              % for blackboard bold, etc
\usepackage{amsthm}                % better theorem environments
\usepackage{enumitem}
\usepackage{amssymb}
\usepackage{hyperref}
\usepackage{tabularx}
\usepackage{amsmath}
\usepackage{mathtools}

\usepackage[mathscr]{euscript}
\usepackage[utf8]{inputenc}
\usepackage[english]{babel}
\usepackage{multicol}

\usepackage{tikz}
\usetikzlibrary{arrows}
\usetikzlibrary{matrix}

\tikzset{diagram/.style={matrix of math nodes, inner sep=0pt, row
    sep=#1, column sep=2.5em, text height=1.5ex, text depth=.25ex,
    nodes={inner sep=1ex}}}
\tikzset{diagram/.default=2.5em}
\newcommand\diagram{\path node[diagram]}

\newtheorem{thm}{Theorem}

\newtheorem{lemma}[thm]{Lemma}
\newtheorem{cor}[thm]{Corollary}
\newtheorem{conj}{Conjecture}
\theoremstyle{definition}
\newtheorem{defn}{Definition}

\newtheorem{rmk}{Remark}

\newtheorem*{related works}{Related Works}

\newtheorem*{question*}{Question}

\theoremstyle{definition}
\newtheorem{question}{Question}

\numberwithin{equation}{section}

\newcommand{\cf}{C_f}

\newcommand{\x}{\mathcal{X}}

\newcommand{\pgl}{\mathrm{PGL}_3(\CC)} 
\newcommand{\ZZ}{\mathbb{Z}} 
\newcommand{\RR}{\mathbb{R}} 
\newcommand{\CC}{\mathbb{C}} 
\newcommand{\PP}{\mathbb{C}P} 
\newcommand{\pn}{\mathrm{PConf}^{n}\cf}
\newcommand{\cn}{\mathrm{UConf}^{n}\cf}

\newcommand{\gf}{\Gamma_f} 
\newcommand{\pt}{\tilde{\psi}} 
\newcommand{\ft}{\tilde{\phi}} 
\newcommand{\gt}{\tilde{\gamma}} 
 
\newcommand{\s}{\sigma}

\newcommand{\ld}{\lambda}

\def\mul#1#2{\ensuremath{\left(\kern-.3em\left(\genfrac{}{}{0pt}{}{#1}{#2}\right)\kern-.3em\right)}}

\begin{document}

\title{Obstructions to choosing distinct points on cubic plane curves}

\author{Weiyan Chen}
%\date{}

\begin{abstract}
Every smooth cubic plane curve has 9 inflection points, 27 sextatic points, and 72 ``points of type nine". Motivated by these classical algebro-geometric constructions, we study the following topological question: Is it possible to continuously choose $n$ distinct unordered points on each smooth cubic plane curve for a natural number $n$? This question is equivalent to asking if certain fiber bundle admits a continuous section or not. We prove that the answer is no when $n$ is not a multiple of $9$. Our result resolves a conjecture of Benson Farb.
\end{abstract}

\maketitle

%Classical constructions from algebraic geometry give a positive answer to the question for various different value of $n$. In this paper, we provide topological obstructions. 

%When $n=9,27$ or 72, the answer to the question is yes, by taking those points to be the inflection points, sextatic points, or points of type 9, respectively. In this paper, we show that the answer is no if $n$ is not divisible by 9. 

\section{Introduction}

%\subsection{The question and its motivation} 
A cubic plane curve in $\CC{{P}}^2$ is the zero locus of a homogeneous  polynomial $f(x,y,z)$ of degree 3. It has been a classical topic to study certain special points on smooth cubic curves, such as the 9 inflection points (dating back at least to Maclaurin; see the Introduction in \cite{d} for an account of the history of this topic), the 27 sextatic points (studied by  Cayley \cite{cayley}), and the 72 ``points of type 9" (studied by  Gattazzo \cite{g}). %Motivated by these constructions in algebraic geometry, Benson Farb asked the following question
 Inspired by these classical constructions, Benson Farb  (private communication) asked the following question: For what integer $n$ is it possible to continuously choose $n$ distinct points on each cubic plane curve as the curve varies in families?

To make the question precise, let $\x$ denote the parameter space of smooth cubic plane curves:
$$\x:=\{f(x,y,z)\ :\ f \text{ is a homogeneous polynomial of degree 3 and is smooth}\}/\sim$$
where $f\sim \ld f$ for any $\ld\in\CC^\times$. There is a fiber bundle whose fiber over $f\in\x$ is  $\cn$, the configuration space of $n$ distinct unordered points on the cubic plane curve $\cf$ defined by $f=0$. %(see Section \ref{set up} for its precise definition).
\begin{equation}\label{bundle}
\begin{tikzpicture}
\diagram (m)
{\cn & E_n\\
& \x\\};
\path [->] (m-1-1) edge node [above] {} (m-1-2)
      (m-1-2) edge node [right] {$\xi_n$} (m-2-2);
\end{tikzpicture}
\end{equation}
\begin{question}[Farb]
\label{que}
For which values $n$ do the bundles $\xi_n$ admit continuous sections?
\end{question}

%\subsection{Previous results: construction}
The aforementioned algebraic constructions and their generalizations give continuous sections to $\xi_n$ for various $n$:
\begin{thm}[Maclaurin, Cayley, Gattazzo]
\label{previous}
The bundle $\xi_n$ has a continuous section when $n=9\sum_{k\in S} J_2(k)$ where $J_2$ is  Jordan's 2-totient function and $S$ is an arbitrary finite set of positive integers, for example, when $n=9, 27, 36, 72, 81, 99, 108, 117, 135, 144, 180...$

%$n/9$ is a sum of Jordan's 2-totient functions, for example, when $n/9=1, 3, 4, 8, 9, 11, 12, 13, 15, 16, 20, 21, 23, 24,...$
\end{thm}

See Section \ref{set up} for more discussion on Theorem \ref{previous} and a definition of  Jordan's 2-totient function. 

In contrast to the various algebraic constructions of sections of $\xi_n$, it was previously unknown whether there is any $n$ such that $\xi_n$ does not have any continuous section. % Farb made the following conjecture: 
\begin{conj}[Farb]
\label{conj}
$n=9$ is the smallest value for $\xi_n$ to have a continuous section.
\end{conj}
Behind this conjecture is the following speculation: the classical algebraic constructions should be the only possible continuous sections. The minimal number comes from the 9 inflection points. After Conjecture \ref{conj}, Farb made a much stronger conjecture: the only continuous sections of $\xi_n$ (up to homotopy) are those given by the algebraic constructions in Theorem \ref{previous}.

We will prove Conjecture \ref{conj} in this paper. In fact, we will prove the following stronger statement.
\begin{thm}
\label{no section}
$\xi_n$ has no continuous section unless $n$ is a multiple of 9.
\end{thm}
Let us remark on the significance of Theorem \ref{no section}: it tells us that in those case when $9\nmid n$ and when our current knowledge since Maclaurin had failed to identify any natural structure of $n$ special points on smooth cubic curves, there is in fact none.

Notice that $\xi_1$ is precisely the tautological bundle whose fiber over every curve in $\x$ is the curve itself. Theorem \ref{no section} thus implies:
\begin{cor}\label{no identity}
The tautological bundle $\xi_1$ does not have any continuous section.
% Let $E_1:=\{(f,x)\in\x\times\PP^2: x\in\cf\}$ be the incidence variety. The tautological bundle \[
% \begin{tikzpicture}
% \diagram (m)
% {\cf & E_1\\
% & \x\\};
% \path [->] (m-1-1) edge node [above] {} (m-1-2)
%       (m-1-2) edge node [right] {$\xi_1$} (m-2-2);
% \end{tikzpicture}
% \] 
% defined by projecting to the first coordinate does not have a continuous section.

\end{cor}
Corollary \ref{no identity} has the following interpretation: it is not possible to continuously choose an elliptic curve structure for every cubic plane curve $\cf$ in the parameter space $\x$, because it is not possible to continuously choose a point on $\cf$ to serve as the identity.

Between the known algebraic constructions (Theorem \ref{previous}) and the topological obstructions presented here (Theorem \ref{no section}), the smallest $n$ for which we don't know whether a continuous section to $\xi_n$ exists or not is $n=18$. 

%has the following interpretation. Each smooth cubic plane curve is an elliptic curve after choosing a point for identity. However, 

%

\section*{Acknowledgement}
I am grateful to Benson Farb for sharing his extremely interesting question and conjecture. I thank Igor Dolgachev for pointing me to the papers by Cayley and Gattazzo, and thank Ronno Das for making me aware of a mistake in the proof of Lemma \ref{ab} in an early draft.

\section{Background}\label{set up}
\subsection{The bundle $\xi_n$}
We will first give more details about the fiber bundle $\xi_n$. Let $\pn$ denote the configuration space of $n$ distinct \emph{ordered} points on $\cf$: 
\begin{align*}
&\pn :=\{(x_1,x_2,\cdots, x_n)\in(\cf)^n\ :\  x_i\ne x_j \ \forall i\ne j\}.
\end{align*}
Define
$$\widetilde{E_n}:=\{(f,x_1,x_2,\cdots, x_n)\in \x\times (\PP^2)^n\ :\ x_i\in \cf\  \forall i \text{ and } x_i\ne x_j \ \forall i\ne j\}.$$
The projection $\widetilde{E_n}\to\x$ gives a fiber bundle:
\[
\begin{tikzpicture}
\diagram (m)
{\pn & \widetilde{E_n}\\
& \x\\};
\path [->] (m-1-1) edge node [above] {} (m-1-2)
      (m-1-2) edge node [right] {$\widetilde{\xi}_n$} (m-2-2);
\end{tikzpicture}
\]
The symmetric group $S_n$ acts freely on $\widetilde{E_n}$ by permuting the $n$ coordinates. Define $E_n:=\widetilde{E_n}/S_n$. Since the bundle projection $\widetilde{\xi}_n: \widetilde{E_n}\to \x$ is invariant of $S_n$, it descends to another bundle $\xi_n: E_n\to\x$ as in (\ref{bundle}). The fiber of $\xi_n$ over $f\in\x$ is precisely the \emph{unordered} configuration space of $\cf$
$$\cn:=\pn/S_n.$$

\begin{rmk}
 The Lie group $\pgl = \mathrm{GL}_n(\CC)/\CC^\times$ acts on $\PP^2$ by projective linear maps, and thus acts on $\x$ by projective linear change of coordinates. Hence, $\pgl$ also acts on $\widetilde{E_n}$ diagonally; this action projects to a $\pgl$-action on $E_n$. The bundle projection $\xi_n:E_n\to\x$ is $\pgl$-equivariant. However, a continuous section $s$ of the bundle $\xi_n$ does not have to be $\pgl$-equivariant. The $\pgl$-action on $\x$ will play a key role in the proof of Theorem \ref{no section}.

\end{rmk}

\subsection{Algebraic sections of $\xi_n$} Every smooth cubic plane curve $\cf$ has 
\begin{enumerate}
\item[(a)] 9 inflection points where the tangent line intersect $\cf$ with multiplicity $3$,
\item[(b)] 27 sextatic points where an irreducible conic intersects $\cf$ with multiplicity $6$ (\cite{cayley}),
\item[(c)] 72 points of type 9 where an irreducible cubic intersects $\cf$ with multiplicity $9$ (\cite{g}).
\end{enumerate}
More generally, Gattazzo \cite{g} defined \emph{points of type $3k$} on $\cf$ to be  points where an irreducible curve of degree $k$ intersects $\cf$ with multiplicity $3k$. In particular,  inflection points and sextatic points are precisely points of type 3 and 6, respectively. %Choosing the points on type $3k$ on $\cf$ gives a section of $\xi_n$ for various different $n$. 
A proof of the following fact can be found on page 392 in \cite{c}: $3k$ points $P_i$ for $i=1,...,3k$ on a cubic curve $C$ are on another curve of degree $k$ if and only if $\sum_{i=1}^{3k} P_i=0$ on $C$ as an elliptic curve with an inflection point as identity.  Therefore, points of type $3k$ on $C$ are precisely the $3k$-torsion points on $C$ that are not $3j$-torsion points for any $j<k$, with an inflection point as identity. This fact was also proved in Corollary 4.3 in \cite{g}. 
 \begin{defn}[Jordan's $2$-totient function]
  $J_2(k)=\text{ number of elements in $(\ZZ/k\ZZ)^2$ of order $k$}.$
 \end{defn}
Jordan's $2$-totient function can be computed using the following formula:
$$J_2(k)= k^2\prod_{p|k, \text{ prime}}(1-\frac{1}{p^2}).$$
It follows from the discussion above that 
$$9 J_2(k) = \text{the number of points of type $3k$ on a cubic curve}$$
The coordinates of the points of type $3k$ on $\cf$ change continuously as we vary $f$ in the parameter space $\x$. For example, the 9 inflections points are precisely the vanishing locus of the Hessian function, which depends continuously on the defining equation of the cubic curve. Therefore, the points of type $3k$ define a continuous section of $\xi_n$ for $n=9 J_2(k)$. More generally, given a finite set of positive integers $S$, the points of type $3k$ for $k\in S$ define a continuous section of $\xi_n$ for $n=9\sum_{k\in S} J_2(k)$, which gives Theorem \ref{previous} in the Introduction.

\section{Proof of Theorem \ref{no section}}
Suppose the bundle $\xi_n$ has a continuous section $s:\x\to E_n$. \\

\noindent\textbf{Step 1: $\pmb{s}$ induces a continuous map $\pmb{\phi: \pgl\to\cn}$}

Throughout the proof, we will fix a basepoint $f\in\x$ to be the Fermat cubic:
$$f(x,y,z)=x^3+y^3+z^3.$$
If we choose a different basepoint in $\x$, the argument will go through with only small modification. See Remark \ref{basepoint} for an explanation. Recall from Section \ref{set up} that $\pgl$ acts on both $E_n$ and $\x$. Define a map 
\begin{align*}
\phi: \pgl &\to E_n\\
g&\mapsto g\cdot s(g^{-1}\cdot f)
\end{align*}
where again $E_n$ is the total space of the fiber bundle $\xi_n$. Notice that for any $g\in\pgl$, we have
\begin{align*}
\xi_n(\phi(g))&=\xi_n(g\cdot s(g^{-1}\cdot f))\\
&= g\cdot\xi_n( s(g^{-1}\cdot f))&\text{since $\xi_n$ is $\pgl$-equivariant}\\
&=g\cdot (g^{-1}\cdot f) &\text{since $\xi_n\circ s=id_\x$}\\
&=f.
\end{align*}
Therefore, the image of $\phi$ is entirely in the fiber of $\xi_n$ over $f$. Hence, we will simply consider $\phi$  as a map  $\phi:\pgl\to \cn$. 

Both the domain and the codomain of $\phi$ are acted upon by a finite group $\gf$,  the group of projective linear automorphisms of the Fermat cubic curve $\cf$: 
$$\gf:=\{g\in\pgl\ :\ g\cdot f=f\in\x\}.$$
$\gf$ acts on $\cn$ since it acts on the curve $\cf$. $\gf$ is a subgroup of $\pgl$ and thus acts on $\pgl$ via multiplication from the left.
\begin{lemma}\label{equi}
The map $\phi:\pgl\to\cn$ is $\gf$-equivariant.
\end{lemma}
\begin{proof}
For any $\gamma\in\ \gf$ and any $g\in\pgl$, we have 
\begin{align*}
\phi(\gamma g) &= (\gamma g)\cdot s(g^{-1}\gamma^{-1}\cdot f)\\
&=\gamma\cdot \Big(g\cdot s(g^{-1}\cdot f)\Big)&\text{since $\gamma^{-1}\in \gf$ fixes $f$}\\
&=\gamma\cdot\phi(g)
\qedhere
\end{align*}
\end{proof}

Observe that $\pi_1(\pgl)\cong\ZZ/3\ZZ$ because we have $\pgl\cong \mathrm{PSL}_3(\CC)$ which is $\mathrm{SL}_3(\CC)$ modulo third roots of unity. On the other hand, Fadell-Neuwirth  (Corollary 2.2 in \cite{fn}) proved that the ordered configuration space $\pn$ (defined in Section \ref{set up}) is aspherical. The same is true for $\cn$. So $\pi_1(\cn)$ must be torsion free because the Eilenberg--MacLane space of any group with torsion must have infinite dimensions. Therefore,  $\phi: \pgl\to\cn$ must induce a trivial map on fundamental groups. 
%Recall as in Section \ref{set up}, $\pn$ denotes the configuration space of $n$ distinct \emph{ordered} points on $\cf$. We have $\cn=\pn/S_n$ where $S_n$ acts on $\pn$ by permuting the ordering. 
 By lifting criterion, the map $\phi$ can be lifted to a map $\ft$ making the following diagram commute:
\[
\begin{tikzpicture}
\diagram (m)
{ & \pn\\
\pgl& \cn\\};
\path [->] (m-2-1) edge node [above] {$\ft$} (m-1-2)
edge node [below] {$\phi$} (m-2-2)
      (m-1-2) edge node [right] {$S_n$} (m-2-2);
\end{tikzpicture}
\]
The natural $\gf$-action on $\pn$ commutes with the $S_n$ action. \\

\noindent\textbf{Step 2: $\pmb{\phi}$ induces a group homomorphism $\pmb{\rho: \gf\to S_n}$.}

The lift $\ft$ may not be $\gf$-equivariant, but it projects down to a $\gf$-equivariant map $\phi$ as in Lemma \ref{equi}. Therefore, for any $\gamma\in \gf$, there exists a unique permutation $\s_\gamma\in S_n$ such that $\ft(\gamma)= \s_\gamma(\gamma\cdot \ft(1))$. Let $\rho:\gf\to S_n$ denote the function $\gamma\mapsto \sigma_\gamma$. 
\begin{lemma}
\label{fail}
\begin{enumerate}
\item For any $\gamma\in\gf$ and any $h\in \pgl$, we have 
\begin{equation}
\label{base}
\ft(\gamma h) =\s_\gamma(\gamma\cdot \ft(h)).
\end{equation}
\item The function $\rho: \gf\to S_n$ is a group homomorphism.
\end{enumerate}
\end{lemma}
\begin{proof}
(1) Since $\pgl$ is connected, we can take a path $p: [0,1]\to \pgl$ such that $p(0)=1$ and $p(1)=h$. Now   $\ft(\gamma\cdot p(t))$ and $\s_\gamma(\gamma\cdot \ft(p(t)))$ are paths in $\pn$ that both lift the path $\phi(\gamma\cdot p(t))$ in $\cn$, with the same starting point $\ft(\gamma)= \s_\gamma(\gamma\cdot \ft(1))$ when $t=0$. Thus they must end at the same point by the uniqueness of path lifting. 

(2) Take any $\beta,\gamma\in\gf$. On one hand, $$\ft(\beta\gamma) = \s_{\beta\gamma}(\beta\gamma\cdot \ft(1)).$$ On the other hand, by (\ref{base})
\begin{align*}
\ft(\beta\gamma) = \s_{\beta}(\beta\cdot \ft(\gamma)) = \s_\beta\Big(\beta\cdot \s_\gamma\big(\gamma\cdot\ft(1)\big)\Big) = \s_\beta\s_\gamma (\beta\gamma\cdot \ft(1))
\end{align*}
where the last equality follows because the $S_n$-action commutes with the $\gf$-action on $\pn$. 
Therefore, $\s_\beta\s_\gamma=\s_{\beta\gamma}$.
\end{proof}

It follows from Lemma \ref{fail} that the homomorphism $\rho$ is trivial if and only if the lift $\ft$ is $\gf$-equivariant. Therefore, the group homomorphism $\rho:\gf\to S_n$ measures the failure of the lift $\ft$ to be $\gf$-equivariant. 

\begin{lemma}
\label{fixed}
Let $\mathrm{Proj}_i:\pn\to \cf$ denote the projection onto the $i$-th coordinate. Consider the action of $\gf$ on $\{1,2,\cdots,n\}$ given by $\rho:\gf\to S_n$. For any $\gamma\in \gf$ and any $i\in \{1,2,\cdots,n\}$, if $\s_\gamma$ fixes $i$, then the composition $\mathrm{Proj}_i\circ \ft$
$$\pgl\xrightarrow\ft\pn\xrightarrow{\mathrm{Proj}_i} \cf$$
 is equivariant with respect to $\gamma$. 
\end{lemma}
\begin{proof}
Given any $h\in\pgl$,  Lemma \ref{fail} gives that $\ft(\gamma\cdot h) =\s_\gamma(\gamma\cdot \ft(h))$. Thus, whenever $\s_\gamma(i)=i$, we have
$$\mathrm{Proj}_i\circ \ft(\gamma\cdot h)=\mathrm{Proj}_i\bigg( \s_\gamma\Big( \gamma\cdot \ft( h)\Big)\bigg) =  \mathrm{Proj}_{i}\Big( \gamma\cdot \ft( h)\Big)=  \gamma\cdot\mathrm{Proj}_{i}\Big( \ft( h)\Big).\qedhere$$
\end{proof}

\begin{rmk}[\textbf{Interpretation of the group homomorphism $\rho$ as monodromy}]
The above construction  of the group homomorphism $\rho$  from a continuous section $s$ appears to be {ad hoc}. We now briefly sketch a more conceptual but less direct construction of $\rho$ to illustrate that $\rho$ is a natural object to consider for finding obstructions to sections. We do not need this equivalent construction of $\rho$ anywhere in the proof. 

A continuous section $s$ produces a cover of $\x$ of degree $n$, where the fiber over each curve is the set of $n$ points on the curve chosen by $s$. The monodromy representation of this cover gives a group homomorphism
$$\pi_1(\x)\to S_n.$$
By restricting the monodromy representation to the $\pgl$-orbit of the Fermat curve
$$\pgl/\gf\approx \pgl\cdot f\hookrightarrow\x$$   
we obtain a group homomorphism 
$$\pi_1(\pgl/\gf)\to S_n.$$
This homomorphism factors through the quotient $\pi_1(\pgl/\gf)\twoheadrightarrow \gf$,
making the following diagram commute:
\[
\begin{tikzpicture}
\diagram (m)
{\pi_1(\pgl/\gf) & S_n\\
\gf & \\};
\path [->] (m-1-1) edge node [above] {} (m-1-2)
              edge node [below] {} (m-2-1)
      (m-2-1) edge [dashed] node [below] {} (m-1-2);
\end{tikzpicture}
\]
The induced homomorphism $\gf\to S_n$ is conjugate to $\rho$.

% where the base point of $\x$ is the Fermat cubic curve $f$.  
% Moreover, the 
% Furthermore, we have 
% $$\pi_1(\pgl/\gf)/\pi(\pgl)\cong \gf$$
\end{rmk}

\noindent\textbf{Step 3: $\pmb{n}$ must be a multiple of $\pmb 9$.} The homomorphism $\rho:\gf\to S_n$ gives an action of $\gf$ on $\{1,2,...,n\}$. In the final step, we will show that $\gf$ contains a subgroup $K$ of order 9 that  acts freely. Thus, $\{1,2,...,n\}$ is a disjoint union of orbits of $K$, each of size 9. In particular, $n$ must be a multiple of 9.

\begin{lemma}
\label{ab}
There are two commuting elements $a$ and $b$ in $\gf$ such that
\begin{enumerate}
\item the subgroup $K:=\langle a,b\rangle\le\gf$ is isomorphic to $\ZZ/3\ZZ\times\ZZ/3\ZZ$.
\item $a$ and $b$ act on $\cf$ as translations of order 3.
\end{enumerate}
\end{lemma}
\begin{rmk}
Lemma \ref{ab} was mentioned without proof in \cite{d}, first paragraph of Section 4. We include a brief proof here to make this paper complete and self-contained.
\end{rmk}
\begin{proof}
Let $a$  and $b$ be two elements in $\pgl$ that are represented by the following two matrices:
\[
a=
\begin{bmatrix}
     0 & 0 & 1 \\
    1 & 0 & 0 \\
    0 & 1 & 0 \\    
\end{bmatrix}
,\ \ \  b=
\begin{bmatrix}
     1 & 0 & 1 \\
    0 & e^{2\pi i/3} & 0 \\
    0 & 0 & e^{4\pi i/3} \\    
\end{bmatrix} 
.\]
The commutator of the two matrices is in the center of $\mathrm{GL}_3(\CC)$. Thus, $a$ and $b$  commute in $\pgl$. Both elements are of order 3 and preserve the Fermat cubic $f:x^3+y^3+z^3$. Thus, (1) is established.  

We first claim that $a$ acts on $\cf$ via translation. In general, every morphism of an elliptic curve is a composition of a translation and a group homomorphism. Thus, it suffices to show that the induced map $a_*\in\mathrm{Aut}\big(H_1(\cf,\ZZ)\big) = \mathrm{SL}_2(\ZZ)$
is identity. A straightforward computation shows that the automorphism $a$ on $\cf$ does not have any fixed point. Thus, by the Lefschetz fixed point theorem, $a_*$ must have trace 2. Now $a_*$ is a finite order element in $\mathrm{SL}_2(\ZZ)$ of trace 2, and therefore must be identity. 

The exact same argument applies to $b$. 
\end{proof}

\begin{lemma}
\label{torus}
For any $\gamma\in K=\langle a,b\rangle$, if there is a continuous $\gamma$-equivariant map $\psi:\pgl\to \cf$, then $\gamma$ must be the identity.
\end{lemma}

\begin{proof}
Note that $\psi: \pgl\to\cf$ must induce a  trivial map on fundamental groups since $\pi_1(\pgl)\cong\ZZ/3\ZZ$ and $\pi_1(\cf)\cong\ZZ^2$. Thus we can lift $\psi$ to $\pt$ making the following diagram commute:
\[
\begin{tikzpicture}
\diagram (m)
{ & \RR^2\\
\pgl& \cf\\};
\path [->] (m-2-1) edge node [above] {$\pt$} (m-1-2)
edge node [below] {$\psi$} (m-2-2)
      (m-1-2) edge node [right] {$p$} (m-2-2);
\end{tikzpicture}
\]
$\gamma\in K$ is a translation of the torus $\cf$. Let $\gt:\RR^2\to\RR^2$ be the unique translation of $\RR^2$ such that
$$\gt(\pt(1))=\pt(\gamma).$$
We claim that for any $h\in \pgl$, 
\begin{equation}
\label{equi lift}
\gt(\pt(h))=\pt(\gamma h).
\end{equation}
Take a path $\mu(t)$ in $\pgl$ such that $\mu(0)=1$ and $\mu(1)=h$. The two paths $\gt(\pt(\mu(t))$ and $\pt(\gamma\cdot\mu(t))$ are both lifts of the path $\gamma(\psi(\mu(t)))=\psi(\gamma\cdot \mu(t))$ starting at the same point $\gt(\pt(1))=\pt(\gamma)$. Thus they must end at the same point by the uniqueness of path lifting, giving that
$\gt(\pt(h))=\pt(\gamma h).$

Applying (\ref{equi lift}) three times, we have
$$\gt^3(\pt(1))=\pt(\gamma^3)=\pt(1).$$
Thus, $\gt^3$ must be the trivial translation of $\RR^2$, and so is $\gt$. In this case (\ref{equi lift}) gives that $\pt(1)=\pt(\gamma)$, and thus $\psi(1)=\psi(\gamma)=\gamma\cdot \psi(1)$. $\gamma$ is a translation of the torus $\cf$ that has a fixed point $\psi(1)$, and therefore must be the identity map.
\end{proof}

Recall that the group homomorphism $\rho:\gf\to S_n$ gives an action of $\gf$ on $\{1,2,\cdots, n\}$. Now Lemma \ref{fixed} and Lemma \ref{torus} together imply that the subgroup $K\le \gf$ acts freely on $\{1,2,\cdots, n\}$. Thus, $\{1,2,\cdots, n\}$ decomposes to a disjoint union of $K$-orbits of size 9. In particular, $n$ must be a multiple of 9. This concludes the proof of Theorem \ref{no section}.

\begin{rmk}[\textbf{What if we choose a different basepoint?}]
\label{basepoint}
The choice of the Fermat cubic curve $f(x,y,z)=x^3+y^3+z^3$ as a basepoint of $\x$ allows us to explicitly write down matrices for $a$ and $b$ in Lemma \ref{ab}.  If a different basepoint $h\in\x$ is chosen, the argument will go through with only minor modification. 
There is an element $g\in\pgl$ that brings $h$ to be in the \emph{Hesse form}: $x^3+y^3+z^3+\lambda xyz$ for some $\lambda\in\CC$ (see e.g. Lemma 1 in \cite{d}). One can check that $a$ and $b$ constructed in Lemma \ref{ab}  act as translations on any smooth cubic curve in the Hesse form. Thus $gag^{-1}$ and $gbg^{-1}$ act on $C_h$ as translations. Now the argument goes through by replacing $a$ and $b$ by their $g$-conjugates. 
\end{rmk}

\end{document}